\documentclass{ifacconf}

\usepackage{natbib,comment}

\usepackage{amsmath,ascmac,url,amsfonts,bm,here,algorithmic,algorithm,amsthm,color,booktabs}

\usepackage{newtxtext,newtxmath}

\usepackage[dvipdfmx]{graphicx}
\usepackage{empheq,enumerate}
\usepackage[format=hang,labelsep=quad,margin=5pt]{caption}
\usepackage{comment}

\newcommand{\argmax}{\mathop{\rm arg~max~}\limits}

\newcommand{\PAR}[1]{{\left( #1 \right)}}

\newcommand{\nn}{\nonumber}

\newcommand{\diag}[1]{{\rm diag}\PAR{#1}}

\newtheorem{theorem}{Theorem}

\newtheorem{prob}{Problem}
\newtheorem{lemma}{Lemma}

\newcommand{\N}{\mathbb{N}}

\newcommand{\R}{\mathbb{R}}

\newcommand{\Tr}[1]{\mathrm{Tr}\left(#1\right)}

\newcommand{\any}{\quad^{\forall}}

\newcommand{\pd}{\partial}

\newcommand{\Dl}[2]{\frac{\pd #1}{\pd #2}}

\newcommand{\st}{\mathrm{s.t.}\text{ ~ ~}}

\newcommand{\V}{\mathcal{V}}
\newcommand{\I}{\mathcal{I}}

\renewcommand{\figurename}{Fig. }

\begin{document}
\begin{frontmatter}
%
\title{Matrix Pontryagin principle approach to controllability metrics maximization under sparsity constraints\thanksref{footnoteinfo}}

\thanks[footnoteinfo]{This work was supported in part by JSPS KAKENHI under Grant Number JP18H01461 and 	JP21H04875.}

\author[kuamp]{Tomofumi Ohtsuka (Kyoto University),} 
\author[Kitakyu]{Takuya Ikeda (The University of Kitakyushu),} 
\author[kuamp]{Kenji Kashima (Kyoto University)} 

\address[kuamp]{Graduate School of Informatics, Kyoto University, Kyoto, Japan.\\\ (e-mail: otsuka.tomofumi@bode.amp.i.kyoto-u.ac.jp, kk@i.kyoto-u.ac.jp)}
\address[Kitakyu]{Faculty of Environmental Engineering, The University of Kitakyushu, Fukuoka, Japan. (e-mail: t-ikeda@kitakyu-u.ac.jp)}

%
%


\maketitle

\begin{abstract}
Controllability maximization problem under sparsity constraints is a node selection problem that selects inputs that are effective for control in order to minimize the energy to control for desired state. In this paper we discuss the equivalence between the sparsity constrained controllability metrics maximization problems and their convex relaxation. The proof is based on the matrix-valued Pontryagin maximum principle applied to the controllability Lyapunov differential equation.
\end{abstract}

\begin{keyword}%
Sparse optimal control, node selection problem, controllability maximization
\end{keyword}

\end{frontmatter}

\section{Introduction}
Sparse optimal problems have attracted a lot of attention in the field of optimal control. Such an approach is useful to find a small number of essential information that is closely related to the control performance of interest, and it is applied widely, for example, \cite{rebalancing}.
This paper investigates the application of sparse optimization to controllability maximization problem, one of the control node selection problems.
The problem is known as the optimization problem minimizing the energy to control for the desired state.  


These problems are generally formulated as maximization of some metric of the controllability Gramian with $L^0/l^0$ constraints, but it is known that the problems include combinatorial structures.
To circumvent this, relaxed problems, where the $L^0/l^0$ norms are replaced by the $L^1/l^1$ norms, are considered for its computational tractability. 
Then, the problem is how to prove the equivalence between the main problem and its relaxation.
The paper \cite{trace} proved the equivalence when the trace of controllability Gramian is adopted as the metric, but its usefulness as a metric is questionable since the designed Gramian may include the zero eigenvalue, so the trace metric does not automatically ensure the controllability. 
The paper \cite{approximation} considered the minimum eigenvalue and the determinant of the controllability Gramian which is useful as metrics, but it avoided the proof of equivalence because of the difficulty and treated approximation problems that are easy to prove the equivalence.
In view of this, this paper newly proposes a method to prove the equivalence for general metrics of controllability. 
Specifically, we adopted the controllability Lyapunov differential equation.
The controllability Lyapunov differential equation is a matrix-valued differential equation whose solution is the controllability Gramian.
By considering the optimal control problem for this Lyapunov differential equation, we can strictly treat useful metrics that are related to the controllability Gramian.

The remainder of this paper is organized as follows. Section \ref{sec:preliminary} provides mathematical preliminaries. Section \ref{sec:analysis} formulates our node scheduling problem using controllability Lyapunov differential equation, and gives a sufficient condition for the main problem to boil down to the relaxation problem. Section \ref{sec:conclusion} offers concluding remarks.

\subsection*{Notation}\label{notation}
We denote the set of all positive integers by $\N$ and the set of all real numbers by $\R$. 
Let $n, m\in\N$.
We denote the zero matrix of size $n\times m$ by $O_{n\times m}$ and the identity matrix of size $n$ by $I_n$.
For any $A,B\in\R^{n\times m}$, 
we denote the Frobenius norm of $A$ by $\|A\|\triangleq\left(\sum_{i=1}^n\sum_{j=1}^m A_{i,j}^2\right)^{1/2}$, and the inner product of $A$ and $B$ by $(A,B)\triangleq\left(\sum_{i=1}^n\sum_{j=1}^m A_{i,j} B_{i,j}\right)$.
Let $C$ be a closed subset of $\R^{n\times m}$ and $A\in C$. A matrix $\Delta\in\R^{n\times m}$ is a \textit{proximal normal} to the set $C$ at the point $A$ if and only if there exists a constant $\sigma\geq0$ such that $(\Delta,B-A)\leq\sigma\|B-A\|^2$ for all $B\in C$. The \textit{proximal normal cone} to $C$ at $A$ is defined as the set of all such $\Delta$, which is denoted by $N_C^P(A)$. We denote the \textit{limiting normal cone} to $C$ at $A$ by $N_C^L(A)$, i.e., $N_C^L(A)\triangleq\{\Delta=\lim_{i\rightarrow\infty} \Delta_i:\Delta_i\in N_C^P(A_i),A_i\rightarrow A,A_i\in C\}$.
For other notations, see~\cite[Section~II]{approximation}.

\section{Preliminary}
\label{sec:preliminary}
Let us consider the following continuous-time linear system
\begin{equation}\label{eq:system}
    \begin{split}
        &\dot{x}(t)=Ax(t)+BV(t)u(t), \quad t\in[0,T],\\
        &V(t)=\diag{v(t)},\quad v(t)\in\{0,1\}^p,
    \end{split}
\end{equation}
where $x(t)\in \R^n$ is the state vector consisting of $n$ nodes, where $x_i(t)$ is the state of the $i$-th node at time $t$; $u(\cdot)\in \R^m$ is the exogenous control input that influences the network dynamics.
Then the controllability Gramian for the system is defined by
\begin{equation}\label{eq:Gramian}
    G_c=\int_0^T e^{A(T-\tau)}BV(\tau)V(\tau)^\top B^\top e^{A^\top(T-\tau)}d\tau.
\end{equation}
We next show why the controllability Gramian is used as the metric of the ease of control. We here recall the minimum-energy control problem:
\begin{equation}\label{eq:unit energy}
    \begin{split}
        \min_u \quad&\int_0^T \|u(t)\|^2dt \\
        \st& \dot{x}(t)=Ax(t)+BV(t)u(t),\\
        & x(0)=0,\quad x(T)=x_f.
    \end{split}
\end{equation}
The minimum control energy is then given by $x_f^\top G_c^{-1}x_f$ (\cite{minimum_unit_energy}). Based on this, recent works have been considered to make $G_c$ as large as possible.
In this paper we design $BV(t)$ in order to maximize some metric of the controllability Gramian.
As the constraints, we introduce $L^0$ and $l^0$ constraints on $v(t)$ to take account of the upper bound of the total time length of node activation and the number of activated nodes at each time.
We consider the following optimal problem that maximizes some metric of $G_c$ under sparsity constraints:
\begin{equation}\label{prob:max_k}
\begin{split}
    \max_v \quad&J(v)=K(G_c) \\
    \st& v(t)\in\{0,1\}^p \any t\in[0,T],\\
    &\|v_j\|_{L^0}\leq \alpha_j \any j\in\{1,2,\dots,p\},\\
    &\|v(t)\|_{l^0}\leq \beta \any t\in[0,T],
\end{split}
\end{equation}
where $K(G_c)$ is a metric of the controllability Gramian, and $\alpha_j>0$ and $\beta>0$ is constant.

Since the maximization problem in (\ref{prob:max_k}) is a combinatorial optimization problem, we consider the following relaxation problem:
\begin{equation}\label{prob:max_k_relaxed}
\begin{split}
    \max_v \quad&J(v)=K(G_c) \\
    \st& v(t)\in [0,1]^p \any t\in[0,T],\\
    &\|v_j\|_{L^1}\leq \alpha_j \any j\in\{1,2,\dots,p\},\\
    &\|v(t)\|_{l^1}\leq \beta \any t\in[0,T].
\end{split}
\end{equation}
This problem is easier to treat than the main problem (especially if $K$ is concave, problem (\ref{prob:max_k_relaxed}) is a convex optimization problem). 
We, however, have to consider the equivalence between the main problem and the corresponding relaxation problem.
\cite{approximation} formulated alternative approximation problem instead of proving the equivalence.
Then this paper proves the equivalence between the main problem and the relaxed one
by using the controllability Lyapunov differential equation.

\section{Proposed method}\label{sec:analysis}

In this section, we formulate a controllability Lyapunov differential equation which holds the controllability Gramian as a solution, and then formulate an optimization problem for a system in which the state space representation is given by the derived differential equation.
We provide an equivalence theorem between the main problem and the corresponding relaxation problem.

\subsection{Problem formulation and relaxation}\label{sec:formulation}

Controllability Lyapunov differential equation is given as follows:
\begin{equation}\label{eq:controllability Lyapunov}
\begin{split}
    &\dot{G_c}(t)=AG_c(t)+G_c(t)A^\top +B V(t)V(t)^\top B^\top,\\
    &G_c(0)=O_{n\times n}.
\end{split}
\end{equation}
Then the controllability Gramian $G_c$ defined by~\eqref{eq:Gramian} corresponds to the solution $G_c(T)$ of (\ref{eq:controllability Lyapunov}) at $t=T$. 
Here we consider the following optimal control problem.
\begin{prob}[Main problem]\label{prob:main}
\begin{equation}
\begin{split}
    \max_v \quad&J(v)=K(G_c(T)) \\
    \st&\dot{G_c}(t)=AG_c(t)+G_c(t)A^\top +B V(t)B^\top,\\ &G_c(0)=O_{n\times n},\\
    & v(t)\in\{0,1\}^p \any t\in[0,T],\\
    &\|v_j\|_{L^0}\leq \alpha_j \any j\in\{1,2,\dots,p\},\\
    &\|v(t)\|_{l^0}\leq \beta \any t\in[0,T].
\end{split}
\end{equation}
\end{prob}
Note that $V(\cdot)V(\cdot)^\top =V(\cdot)$ since $v(\cdot)\in\{0,1\}^p$, so we rewrite the controllability Lyapunov differential equation.
Problem \ref{prob:main} is a combinatorial optimization problem, so we consider the following relaxed problem, where the $L^0/l^0$ norms are replaced by the $L^1/l^1$ norms, respectively.

\begin{prob}[Relaxed problem]\label{prob:relaxed}
\begin{equation}
\begin{split}
    \max_v \quad&J(v)=K(G_c(T)) \\
    \st&\dot{G_c}(t)=AG_c(t)+G_c(t)A^\top +B V(t) B^\top,\\ &G_c(0)=O_{n\times n},\\
    & v(t)\in[0,1]^p \any t\in[0,T],\\
    &\|v_j\|_{L^1}\leq \alpha_j \any j\in\{1,2,\dots,p\},\\
    &\|v(t)\|_{l^1}\leq \beta \any t\in[0,T].
\end{split}
\end{equation}
\end{prob}
In what follows, we suppose that $K$ is continuously differentiable.

\subsection{discreteness and equivalence}\label{subsec:proof}
We define the set of feasible solutions of Problem \ref{prob:main} and Problem \ref{prob:relaxed} by $\V_0$ and $\V_1$, i.e.,
\begin{align*}
    \V_0\triangleq\{v:&v(t)\in\{0,1\}^p\any t,\quad\|v_j\|_{L^0}\leq\alpha_j\any j,\\
    &\quad \|v(t)\|_{l^0}\leq\beta \any t\},\\
    \V_1\triangleq\{v:&v(t)\in[0,1]^p\any t,\quad\|v_j\|_{L^1}\leq\alpha_j\any j,\\
    &\quad \|v(t)\|_{l^1}\leq\beta \any t\}.
\end{align*}
Note that $\V_0\subset\V_1$, since $\|v_j\|_{L^1}=\|v_j\|_{L^0}$ for all $j$ and $\|v(t)\|_{l^1}=\|v(t)\|_{l^0}$ on $[0,T]$ for any measurable function $v$ with $v(t)\in\{0,1\}^p$ on $[0,T]$.
The inclusion is proper in general, since the $L^1$/$l^1$ constraints do not automatically guarantee the $L^0$/$l^0$ constraints and some functions in $\V_1$ are not obviously binary.
Then, we first show the discreteness of solutions of Problem \ref{prob:relaxed}, which guarantees that the optimal solutions of Problem \ref{prob:relaxed} belongs to the set $\V_0$.
For this purpose, we prepare lemmas.

\begin{lemma}[Matrix Pontryagin principle]\label{lem:maximum principle}
Let us consider the following optimization problem
\begin{equation}
\begin{split}
    \min_U \quad&J=L_f(X(T))\\
    \st& \dot{X}(t)=F(X(t),U(t)),\\ 
    &X(0)=X_0,\quad X(T)\in E,\quad U(t)\in \Omega,
\end{split}\label{prob:OC}
\end{equation}
where $L_f$ is continuously differentiable, $F$ is continuous, $D_X F(X(t),U(t))$ is continuous with respect to $t,X,U$, $X(t)\in\R^{n\times m}$, $U(t)\in\R^{p\times q}$, $X_0\in\R^{n\times m}$, $T>0$, $E\subset\R^{n\times m}$, and $\Omega\subset\R^{p\times q}$. 
Note that $(L_f,F,X_0,T,E,\Omega)$ is given.
We define Hamiltonian function $H:\R^{n\times m}\times\R^{n\times m}\times\R^{p\times q}\to\R$ associated to problem (\ref{prob:OC}) by
\begin{equation}
    H(X(t),P(t),U(t))=\Tr{P(t)^\top F(X(t),U(t))}.
\end{equation}
Let the process $(X^*(t),V^*(t))$ be a local minimizer for the problem (\ref{prob:OC}).
Then there exists a matrix $P:[0,T]\rightarrow\R^{n\times m}$, and a scalar $\eta$ equal to 0 or 1 satisfying the following conditions:
\begin{itemize}
    \item the nontriviality condition:
    \begin{equation}
        (\eta,P(t))\neq0\any t\in[0,T],\label{condition:nontrivial}
    \end{equation}
    \item the transversality condition:
    \begin{equation}
        -P(T)\in\eta\nabla L_f(X^*(T))+N_E^L(X^*(T)),\label{condition:transversality}
    \end{equation}
    \item the adjoint equation for almost every $t\in[0,T]$:
    \begin{equation}
        -\dot{P}(t)=D_X H(X^*(t),P(t),U^*(t)),\label{condition:adjoint}
    \end{equation}
    \item the maximum condition for almost every $t\in[0,T]$:
    \begin{equation}
        H(X^*(t),P(t),U^*(t))=\sup_{U\in\Omega} H(X^*(t),P(t),U).\label{condition:maximize}
    \end{equation}
\end{itemize}
\end{lemma}

\begin{proof}
We define a mapping $\psi_{nm}:\R^{n\times m}\to\R^{nm}$ by
\begin{equation}\label{psi}
    \psi_{nm}(X)=
    \begin{bmatrix} X_1^\top, \dots, X_m^\top \end{bmatrix}^\top,
\end{equation}
where $X_i\in\R^n$ denotes the $i$th column of a matrix $X\in\R^{n\times m}$.
From \cite{matrix_minimum_principle}, $\psi_{nm}$ is a regular linear mapping (hence $\psi_{nm}^{-1}$ exists), and preserves the inner product. 
Then problem (\ref{prob:OC}) is equivalent to
\begin{equation}\label{prob:vector}
\begin{split}
    \min_u \quad&J=l_f(x(T))\\
    \st& \dot{x}(t)=f(x(t),u(t)), \\
    &x(0)=x_0\quad x(T)\in e,\quad u(t)\in \omega, 
\end{split}
\end{equation}
where $x=\psi_{nm}(X)$, $u=\psi_{pq}(U)$, and $l_f,f,x_0,e,\omega$ corresponds to $L_f,F,X_0,E,\Omega$ respectively.
We define the Hamiltonian function $h:\R^{nm}\times\R^{nm}\times\R^{pq}\to\R$ associated to problem (\ref{prob:vector}) by
    $h(x(t),p(t),u(t))=p(t)^\top f(x(t),u(t))$
and denote the local minimizer for problem (\ref{prob:vector}) by ($x^*(t)$, $u^*(t)$).
Then there exists an arc $p:[0,T]\rightarrow\R^{nm}$ and a scalar $\eta$ equal to 0 or 1 satisfying the following conditions (Pontryagin's Maximum Principle (\cite{maximum_principle})):
\begin{itemize}
    \item the nontriviality condition:
    \begin{equation}
        (\eta,p(t))\neq0\any t\in[0,T],
    \end{equation}
    \item the transversality condition:
    \begin{equation}
        -p(T)\in\eta\nabla l_f(x^*(T))+N_e^L(x^*(T)),
    \end{equation}
    \item the adjoint equation for almost every $t\in[0,T]$
    \begin{equation}
        -\dot{p}(t)=D_x h(x^*(t),p(t),u^*(t)),
    \end{equation}
    \item the maximum condition for almost every $t\in[0,T]$
    \begin{equation}
        h(x^*(t),p(t),u^*(t))=\sup_{u\in\omega} h(x^*(t),p(t),u).
    \end{equation}
\end{itemize}
Since $\psi_{nm}^{-1}$ exists,
we obtain the Hamiltonian function associated to $h(x^*(t),p(t),u^*(t))$ as follows:
\begin{equation}
    H(X^*(t),P(t),U^*(t))= \Tr{P^\top(t) F(X^*(t),U^*(t))},
\end{equation}
which satisfies (\ref{condition:nontrivial}), (\ref{condition:transversality}), (\ref{condition:adjoint}), and (\ref{condition:maximize}),
where $X^*=\psi_{nm}^{-1}(x^*)$, $U^*=\psi_{pq}^{-1}(u^*)$, $P=\psi_{nm}^{-1}(p)$. This completes the proof.
\end{proof}

\begin{lemma}\label{lem:cone}
Define a set
\begin{equation}
    E\triangleq\{A\in\R^{n\times n}:A_{i,j}\leq\alpha_{i,j},\quad (i,j)\in\I\},
\end{equation}
and fix any $\gamma\in E$, where $\I\subset\N^{n\times n}$ is a set of positions of elements of $A$ for which inequality constraints are given. Then any $\delta\in N_E^L(\gamma)$ satisfies
\begin{align}
    &\delta_{i,j}(\gamma_{i,j}-\alpha_{i,j})=0 \any (i,j)\in\I,\label{eq:cone1}\\
    &\delta_{i,j}\geq 0 \any (i,j)\in\I,\label{eq:cone2}\\
    &\delta_{i,j}=0 \any (i,j)\notin\I.\label{eq:cone3}
\end{align}
\end{lemma}

\begin{proof}
Fix any $\hat{A}\in E$ and $\hat{a}=\psi_{nn}(\hat{A})$, where $\psi_{nn}$ is from (\ref{psi}). 
Then we obtain a set $e$ satisfying $\hat{a}\in e$ as follows:
\begin{equation}
    e\triangleq\{a\in\R^{n^2}:a_j\leq\alpha'_j,\quad j\in\I'\},
\end{equation}
where $\alpha'=\psi_{nn}(\alpha)$ and $\I'\subset\N^{n^2}$ is a set corresponding to $\I$.
Take any $\gamma'\in e$, then we have
\begin{align}
    &\delta'_i (\gamma'_i-\alpha'_i)=0 \any i\in\I',\\
    &\delta'_i \geq 0 \any i\in\I',\\
    &\delta'_i=0 \any i\notin\I'
\end{align}
for all $\delta'\in N_E^L(\gamma')$ (\cite{approximation}).
Finally, we obtain (\ref{eq:cone1}), (\ref{eq:cone2}), (\ref{eq:cone3}) 
where $\delta=\psi_{nn}^{-1}(\delta')$ and $\gamma=\psi_{nn}^{-1}(\gamma')$. 
\end{proof}

\begin{theorem}\label{th:discrete}
    Let $G_c^*(t)$ and $V^*(t)$ be a local optimal solution of Problem \ref{prob:relaxed}. Assume that 
    \begin{equation*}
        q_j(t) \triangleq b_j^\top e^{A^\top(T-t)} \Dl{K(G_c^*(T))}{G_c^*(T)} e^{A(T-t)} b_j
    \end{equation*}
    and $q_i(t) - q_j(t)$ is not constant on $[0,T]$ for all $i,j\in\{1,2,\dots,p\}$. Then any solution to Problem \ref{prob:relaxed} takes only the values in the binary set \{0,1\} almost everywhere.
\end{theorem}

\begin{proof}
We first reformulate Problem \ref{prob:relaxed} into a form to which Lemma \ref{lem:maximum principle} is applicable. 
The value $\|v_j\|_{L^1}$ is equal to the final state $y_j(T)$ of the system
    $\dot{y_j}(t)=v_j(t)$
with $y_j(0)=0$. 
Define $Y(t)\triangleq\diag{y(t)}$ and matrices $X(t)$, $\bar{V}(t)$, $\bar{A}$, $\bar{B}$ by
\begin{equation}
\begin{split}
    X(t)&\triangleq
    \begin{bmatrix}
        G_c(t) & O_{n\times p}\\O_{p\times n} & Y(t)\\
    \end{bmatrix},\;
    \bar{V}(t)\triangleq
    \begin{bmatrix}
        V(t) & O_{p\times p}\\O_{p\times p} & V(t)\\
    \end{bmatrix},\\
    \bar{A}&\triangleq
    \begin{bmatrix}
        A & O_{n\times p}\\O_{p\times n} & O_{p\times p}\\
    \end{bmatrix},\;
    \bar{B}\triangleq
    \begin{bmatrix}
        B & O_{n\times p}\\O_{p\times p} & I_p\\
    \end{bmatrix}.\nn
\end{split}
\end{equation}
Then, Problem \ref{prob:relaxed} is equivalently expressed as follows:
\begin{equation}\label{prob:rewrite}
\begin{split}
    \min_v \quad&J(V)=-L_f(X(T))\\
    \st&\dot{X}(t)=\bar{A}X(t)+X(t)\bar{A}^\top +\bar{B}\bar{V}(t)\bar{B}^\top,\\
    &X(0)=O_{(n+p)\times (n+p)},\quad X(T)\in E,\\
    &v(t)\in\Omega \any t\in[0,T],
\end{split}
\end{equation}
where $L_f(X(T))=K(G_c(T))$, $E=\{X(T): y_j(T)\leq\alpha_j\any j\in\{1,2,\dots,p\}\}$, $\Omega=\{v(t):v(t)\in[0,1]^p,\|v(t)\|_{l^1}\leq\beta\}$.
This is an optimal control problem to which Lemma \ref{lem:maximum principle} is applicable. 
We define the Hamiltonian function $H$ associated to problem (\ref{prob:rewrite}) by
\begin{equation}
    H(X,P,V)=\Tr{P^\top(\bar{A}X(t)+X(t)\bar{A}^\top +\bar{B}\bar{V}(t)\bar{B}^\top)}.\nn
\end{equation}
We define two matrices as follows:
\begin{align}
    X^*(t)\triangleq
    \begin{bmatrix}
        G_c^*(t) & O_{n\times p}\\
        O_{p\times n} & Y^*(t)
    \end{bmatrix},\quad
    \bar{V}^*(t)\triangleq
    \begin{bmatrix}
        V^*(t) & O_{p\times p}\\
        O_{p\times p} & V^*(t)
    \end{bmatrix}.
\end{align}
Then ($X^*(t)$, $\bar{V}^*(t)$) is the local minimizer of problem (\ref{prob:rewrite}) because of the equivalence between Problem \ref{prob:relaxed} and problem (\ref{prob:rewrite}), and there exists a scalar $\eta$ equal to 0 or 1 and a matrix $P:[0,T]\rightarrow\R^{n\times n}$ satisfying the conditions (\ref{condition:nontrivial}), (\ref{condition:transversality}), (\ref{condition:adjoint}), (\ref{condition:maximize}).
It follows from (\ref{condition:adjoint}) that
\begin{equation}
    -\dot{P}(t)=\bar{A}^\top P(t)+P(t)\bar{A},\nn
\end{equation}
which leads to
\begin{equation}\label{eq:P}
\begin{split}
    P(t)&=e^{\bar{A}^\top(T-t)}P(T)e^{\bar{A}(T-t)}\\
    &=\begin{bmatrix}
        e^{A^\top(T-t)}P^{(11)}(T)e^{A(T-t)} & e^{A^\top(T-t)}P^{(12)}(T)\\
        P^{(21)}(T)e^{A(T-t)} & P^{(22)}(T)
    \end{bmatrix},
\end{split}
\end{equation}
where
\begin{equation}
\begin{split}
    P(t)=
    \begin{bmatrix}
        P^{(11)}(t) & P^{(12)}(t)\\
        P^{(21)}(t) & P^{(22)}(t)\\
    \end{bmatrix}
\end{split}
\end{equation}
with
$P^{(11)}(t)\in\R^{n\times n}$ and $P^{(22)}(t)\in\R^{p\times p}$.
Note that
\begin{equation}
    \begin{bmatrix}
        -P^{(11)}(T)+\eta\Dl{K(G_c^*(T))}{G_c^*(T)} & -P^{(12)}(T)\\
        -P^{(21)}(T) & -P^{(22)}(T)
    \end{bmatrix}\in N_E^L (X^*(T))\nn
\end{equation}
by (\ref{condition:transversality}), then we have
\begin{align}
    &P_{j,j}^{(22)}(T)(y_j(T)-\alpha_j)=0 \quad j=\{1,2,\dots,p\},\\
    &P_{j,j}^{(22)}(T)\leq 0 \quad j=\{1,2,\dots,p\},\\
    &P_{i,j}^{(22)}(T) =0 \any (i,j)\in\{(i,j):i\neq j\},\\
    &-P^{(11)}(T)+\eta\Dl{K(G_c^*(T))}{G_c^*(T)}=O_{n\times n}, \\
    &P^{(12)}(T)=O_{n\times p}, \quad
    P^{(21)}(T)=O_{p\times n},
\end{align}
from Lemma \ref{lem:cone}. Substituting these into (\ref{eq:P}), we get
\begin{equation}
    P(t)=\begin{bmatrix}
        e^{A^\top(T-t)}\eta\Dl{K(G_c^*(T))}{G_c^*(T)}e^{A(T-t)} & O_{n\times p}\\
        O_{p\times n} & P^{(22)}(T)
    \end{bmatrix}.
\end{equation}
Then, we have
\begin{equation}
\begin{split}
    \Tr{P^\top(t)\bar{B}\bar{V}(t)\bar{B}^\top}
    &=\Tr{\bar{B}^\top P^\top(t) \bar{B}\bar{V}(t)}\\
    &=\Tr{\begin{bmatrix}
        B^\top {P^{(11)}}^\top(t)B & O_{p\times p}\\
        O_{p\times p} & P^{(22)}(t)\\
    \end{bmatrix}
    \bar{V}(t)}\\
    &=\sum_{j=1}^p \left(\eta q_j(t)+P_{j,j}^{(22)}(T)\right) v_j(t).
\end{split}\nn
\end{equation}
It follows from (\ref{condition:maximize}) that
\begin{equation}\label{eq:max_v}
    v^*(t) = \argmax_{v\in\Omega} \sum_{j=1}^p \left(\eta q_j(t)+P_{j,j}^{(22)}(T)\right) v_j.
\end{equation}
We here claim that $\eta=1$. Indeed, if $\eta=0$, $P^{(22)}(T)\neq O_{p\times p}$ follows from (\ref{condition:nontrivial}), i.e., there exists some $j$ that satisfies 
\begin{equation}\label{eq:eta0}
    P_{j,j}^{(22)}(T)<0,\quad y_j^*(T)=\alpha_j.
\end{equation}
Hence, from (\ref{eq:max_v}) and (\ref{eq:eta0}), we have $v_j^*(t)=0$ for all $t\in[0,T]$, i.e., $y_j^*(T)=\|v_j^*\|_{L^1}=0$. This contradicts to (\ref{eq:eta0}). Thus, $\eta=1$.
From the assumption, it is easy to verify that
\begin{enumerate}[1)]
    \item we have
        $q_j(t)+P_{j,j}^{(22)}(T)\neq 0$
    almost everywhere for all $j=\{1,2,\dots,p\}$,
    \item there exists $j_k$ $:$ $[0,T]\rightarrow\{1,2,\dots,p\}, k=1,2,\dots,p$, such that
    \[
        q_{j_1(t)}(t)+P_{j_1(t),j_1(t)}^{(22)}(T)
        >\cdots
        > q_{j_p(t)}(t)+P_{j_p(t),j_p(t)}^{(22)}(T)
    \]
    almost everywhere.
\end{enumerate}
Hence, we find
\begin{equation}\label{eq:optimal_v}
    v_j^*(t)=\begin{cases}
    1 \quad \mbox{if} \quad j\in\Xi_1(t)\cap \Xi_2(t),\\
    0 \quad \mbox{otherwise}
    \end{cases}
\end{equation}
for almost every $t\in[0,T]$, where
\begin{align*}
    \Xi_1(t)&\triangleq\{j_1(t),j_2(t),\dots,j_\beta(t)\},\\
    \Xi_2(t)&\triangleq\{k\in\{1,2,\dots,p\}:q_{j_k(t)}(t)+P_{j_k(t),j_k(t)}^{(22)}(T)>0\}.
\end{align*}
This completes the proof.
\end{proof}

The following theorem is the main result, which shows the equivalence between Problem \ref{prob:main} and Problem \ref{prob:relaxed}.

\begin{theorem}[equivalence]\label{th:equivalence}
Assume that $q_j(t)$ and $q_i(t) - q_j(t)$ is not constant on $[0,T]$ for all $i,j\in\{1,2,\dots,p\}$. Denote the set of all solutions of Problem \ref{prob:main} and Problem \ref{prob:relaxed} by $\V_0^*$ and $\V_1^*$, respectively. If the set $\V_1^*$ is not empty, then we have $\V_0^*=\V_1^*$.
\end{theorem}

\begin{proof}
Denote any solution of Problem \ref{prob:relaxed} by $\hat{v}\in\V_1^*$. 
It follows from Theorem \ref{th:discrete} that $\hat{v}(t)\in\{0,1\}^p$ almost everywhere. 
Note that the null set $\cup_{j=1}^p\{t\in[0,T]:\hat{v}_j(t)\notin\{0,1\}\}$ does not affect the cost, and hence we can adjust the variables so that $\hat{v}(t)\in\{0,1\}^p$ on $[0,T]$, without loss of the optimality. 
We have
\begin{equation}
    \|\hat{v}(t)\|_{l^1}=\|\hat{v}(t)\|_{l^0},\quad \|\hat{v}_j\|_{L^1}=\|\hat{v}_j\|_{L^0}\nn
\end{equation}
for all $j$. 
Since $\hat{v}\in\V_1$, we have $\|\hat{v}(t)\|_{l^0}\leq\beta$ and $\|\hat{v}_j\|_{L^0}\leq\alpha_j$ for all $t$ and $j$.
Thus, $\hat{v}\in\V_0$. Then,
\begin{equation}
    J(\hat{v})\leq\max_{v\in\V_0}J(v)\leq\max_{v\in\V_1}J(v)=J(\hat{v}),
\end{equation}
where the first relation follows from $\hat{v}\in\V_0$,  the second relation follows from $\V_0\subset\V_1$, and the last relation follows from $\hat{v}\in\V_1^*$.
Hence, we have
\begin{equation}\label{j}
    J(\hat{v})=\max_{v\in\V_0}J(v),
\end{equation}
which implies $\hat{v}\in\V_0^*$. 
Hence, $\V_1^*\subset\V_0^*$ and $\V_0^*$ is not empty.

Next, take any $\tilde{v}\in\V_0^*$. Note that $\tilde{v}\in\V_1$, since $\V_0^*\subset\V_0\subset\V_1$. In addition, it follows from (\ref{j}) that $J(\tilde{v})=J(\hat{v})$. Therefore, $\tilde{v}\in\V_1^*$, which implies $\V_0^*\subset\V_1^*$. This gives $\V_0^*=\V_1^*$.
\end{proof}

\section{Conclusion}\label{sec:conclusion}

In this paper, we discussed the equivalence between the sparsity constrained controllability metrics maximization problems and their convex relaxation. The proof is based on the matrix-valued Pontryagin maximum principle applied to the controllability Lyapunov differential equation. The existence of optimal solutions and computational cost are currently under investigation. 



\bibliography{Mybib_sparse} 

\end{document}